\documentclass[reqno]{amsproc}
\usepackage[utf8]{inputenc}
\usepackage{amsfonts}
\usepackage{graphicx}
\usepackage{amscd}
\usepackage{amsmath}
\usepackage{amssymb}
\usepackage{latexsym}
\usepackage{hyperref}
\usepackage[all]{xy}
\usepackage{color}
\usepackage{mathrsfs}
\theoremstyle{plain}

\newtheorem{corollary}{\bf Corollary}

\newtheorem{example}{\bf Example}
\newtheorem{lemma}{\bf Lemma}

\newtheorem{proposition}{\bf Proposition}
\newtheorem{remark}{\bf Remark}

\newtheorem{theorem}{\bf Theorem}
\numberwithin{equation}{section}

\begin{document}

\title[Compact gradient Einstein-type manifolds with boundary]{Compact gradient Einstein-type manifolds with boundary}

\author[Allan George de Carvalho Freitas]{Allan George de Carvalho Freitas$^1$}
\address{$^1$Departamento de Matemática, Universidade Federal da Paraíba, Cidade Universitária, 58051-900, João Pessoa, Paraíba, Brazil.}
\email{allan@mat.ufpb.br;allan.freitas@academico.ufpb.br}
\urladdr{https://www.ufpb.br}

\author[José Nazareno Vieira Gomes]{José Nazareno Vieira Gomes$^2$}
\address{$^2$Departamento de Matemática, Universidade Federal de São Carlos, Rod. Washington Luís, Km 235, 13565-905, São Carlos, São Paulo, Brazil.}
\email{jnvgomes@ufscar.br}
\urladdr{https://www2.ufscar.br}

\keywords{Einstein-type metrics; Rigidity results; Geodesic balls.}
\subjclass[2010]{Primary 53C25; Secondary 53C24.}

\begin{abstract}
We deal with rigidity results for compact gradient Einstein-type manifolds with nonempty boundary. As a result, we obtain new characterizations for hemispheres and geodesic balls in simply connected space forms. In dimensions three and five, we obtain topological characterizations for the boundary and upper bounds for its area.
\end{abstract}
\maketitle

\section{Introduction}

Throughout this paper, $M^n$ denotes an $n$-dimensional, compact, connected, oriented, smooth manifold with nonempty boundary $\Sigma$ and $n\geq3.$ We say that a Riemannian metric $g$ is a gradient Einstein-type metric on $M^{n}$ or that $(M^n,g)$ is a gradient Einstein-type manifold if there exists a smooth function $u$ on $M^{n}$ satisfying the overdetermined problem
\begin{equation}\label{1-1}
\left\{\begin{array}{rcl}
\nabla^{2}u &=& \frac{\mu}{\beta}u\big(\Lambda g-\frac{\alpha}{\beta}Ric\big)+\gamma g,\\
u&>&  0\quad \hbox{in}\quad int(M),\\
u&=& 0\quad \hbox{on}\quad  \Sigma,
\end{array}\right.
\end{equation}
for some smooth function $\Lambda$ on $M^{n}$ and constants $\alpha, \beta, \mu, \gamma \in \mathbb{R}$, with $\beta\neq 0.$ Here $\nabla^2u$ denotes the Hessian of $u,$ and $Ric$ stands for the Ricci tensor of $g.$ We observe that the $\mu=0$ and $\gamma\neq0$ case has already been covered by the classical Reilly's result~\cite[Lemma~3 and Theorem~B]{reilly2}, from which we know that if a compact Riemannian manifold $(M^n,g)$ with nonempty connected boundary $\Sigma$ admits a smooth function $f$ on $M^n$ and a nonzero constant $L$ such that $\nabla^2f=Lg$ and $f|_\Sigma$ is constant, then it is isometric to a (metric) ball in a Euclidean space, while the $\mu=\gamma=0$ case, the function $u$ must be constant by the maximum principle and we have a contradiction. So, for our purposes, it is enough to consider $\mu\neq0.$ Furthermore, for the $\gamma\neq 0$ case, we have to consider $0$ as a regular value of $u$, while for $\gamma=0$, this is a consequence of \eqref{1-1}, see Proposition~\ref{prop1}.

We highlight that the sign of $\alpha\mu$ is of crucial significance for compact gradient Einstein-type manifolds of constant scalar curvature, see Section~\ref{RbtEtm}. For instance, in dimension three, we give a complete topological classification for the boundary by means of the sign of $\alpha\mu,$ see Section~\ref{3-dimension}.

The class of metrics defined by \eqref{1-1} is closely related to two particular cases of metrics studied in both mathematics and physics frameworks. The first one relies on the theory of static and $V$-static metrics, which have applications on General Relativity, black holes problems and critical metrics of volume functional, see, e.g.,~\cite{hawking,HMRa,MT}. Indeed, for $\alpha=\beta=-\mu$ and $\gamma=0$ or $\gamma=1$, we additionally suppose that $\Lambda$ is a constant and $-\Delta u=\Lambda u$ so that we recover the special cases of static and $V$-static metrics, respectively. The second one arises from the construction of warped product Einstein metrics through the equation for the Ricci tensor of the base space with nonempty boundary, the so-called $(\lambda,n+m)$-Einstein equation by He, Petersen and Wylie~\cite{HPW1}. They observed that such equation is a natural case to consider since a manifold without boundary often occurs as a warped product over a manifold with nonempty boundary.

In the setting of manifolds without boundary, we highlight the general family of gradient Einstein-type metrics $g$, namely:
\begin{equation}\label{1-2}
\alpha Ric + \beta \nabla^2f+ \mu df\otimes df = (\rho S +\lambda)g,
\end{equation}
where $\alpha,\beta,\mu,\rho$ are constants with $(\alpha,\beta,\mu)\neq(0,0,0)$, $f,\lambda$ are smooth functions and $S=tr(Ric)$ is the scalar curvature of $g$. This definition, originally introduced by Catino et al.~\cite{CMMR}, unifies various particular cases of metrics studied in the current literature, such as Ricci solitons, $\rho-$Einstein solitons, Yamabe solitons, among others, see also the paper by the second author in~\cite{nazareno}. A metric as defined in \eqref{1-2} is nondegenerate if $\beta^2\neq(n-2)\alpha\mu$ and $\beta\neq0$. Otherwise, if $\beta^2=(n-2)\alpha\mu$ and $\beta\neq0$ we have a degenerate gradient Einstein-type metric. In~\cite{CMMR} the authors give a justification for this terminology through an equivalence of a degenerate gradient Einstein-type metric with a conformally Einstein metric. In Section~\ref{Sec-DNDC}, we observe that PDE~\eqref{1-1} in $int(M)$ is equivalent to an equation like~\eqref{1-2}, this allows us to characterize the degenerate condition in our case.

In~\cite{nazareno} the second author focused his attention on the case $\beta\neq0$, which includes both degenerate and nondegenerate cases. In particular, the nondegenerate condition has been crucial for him, see~\cite[Eq.~(3.10)]{nazareno}. For it, he used that equation~\eqref{1-2} is equivalent to
\begin{equation}\label{1-3}
\frac{\alpha}{\beta}Ric + \frac{\beta}{\mu u}\nabla^2 u=\Lambda g,
\end{equation}
where $u=e^{\frac{\mu}{\beta}f}$ and $\Lambda=\frac{1}{\beta}(\rho S+\lambda)$. This was indeed one of our key insights to approach this class of metrics on a smooth manifold with nonempty boundary. Here, we assume the existence of a nonnegative solution $u$ of PDE~\eqref{1-1}, and investigate its topological consequences.

We first focus attention on basic properties for the boundary of a gradient Einstein-type manifold and give some examples, see Section~\ref{bpBEx}. After, we address the Einstein case, in this setting, we highlight Proposition~\ref{sinalgamma} as an important inequality by means of $\gamma.$ Furthermore, we prove two interesting rigidity results characterizing geodesic balls and hemispheres, see Theorems~\ref{teogammaneq0} and~\ref{gamma0}. Coming back to the general case, we obtain an integral relationship between the gradient Einstein-type manifold and its boundary, see Proposition~\ref{thmA}. As an application, we prove a Chru$\acute{s}$ciel type inequality and approach its consequences to gradient Einstein-type manifolds, see Theorems~\ref{corovstatic} and~\ref{corostatic}. Also, we obtain upper bounds for the area boundary in arbitrary dimension, see Theorem~\ref{areangeneral}. In particular, we get topological constraints and classifications for the boundary of a gradient Einstein-type manifold in three and five dimensional cases, see Proposition~\ref{topsphere} and Corollaries~\ref{coro1}, \ref{coro3} and~\ref{teofivedimensional}. Finally, Theorem~\ref{HCGEtM} gives a rigidity result of geodesic balls on a class of gradient Einstein-type manifolds with constant Ricci curvatures and nondegenerate condition in their topological interior. 

\section{Boundary properties and some examples}\label{bpBEx}

Here, we discuss basic properties for the boundary of a gradient Einstein-type manifold by means of $\gamma$. We observe that such properties has been proved in particular cases by He-Petersen-Wylie~\cite{HPW1} and Miao-Tam~\cite{MT}. Our proof follows as in the latter two papers. For the sake of clarity, we have adapted the proof to our case.

\begin{proposition}\label{prop1}
Let $(M^n,g)$ be a gradient Einstein-type manifold with boundary $\Sigma$.
\begin{itemize}
\item[(i)] The boundary $\Sigma$ is a totally umbilical hypersurface, with second fundamental form $\mathcal{A}=-\gamma|\nabla u|^{-1}g_{\Sigma}.$ In particular, if $\gamma=0$, then $\Sigma$ is a totally geodesic hypersurface.
\item[(ii)]Let $\Sigma=\cup\Sigma_{i}$ be the disjoint union of all connected components $\Sigma_i$ of $\Sigma$. Then, the restriction of $\xi_i:=|\nabla u|$ to each $\Sigma_i$ is a positive constant.
\end{itemize}
\end{proposition}

\begin{proof}
In the case where $\gamma\neq 0$, we already assume that $0$ is a regular value for $u,$ then $\nabla u(x)\neq 0$ for all $x\in\Sigma$. Let us now assume that $\gamma=0$ and proving that, without an additional hypothesis, $\nabla u(x)\neq 0$ for all $x\in\Sigma$. For it, we take a unit speed geodesic $\sigma:[0,1]\to M^n$ so that $\sigma(0)=x\in\Sigma$ and $\sigma(1)\in int(M)$, thus the function $\theta:[0,1]\to\mathbb{R}$ given by $\theta(t)= u(\sigma(t))$ satisfies $\theta(0)=u(x)=0$ and
\begin{eqnarray*}
\theta''(t) = \nabla^2u(\sigma'(t),\sigma'(t)) =\frac{\mu}{\beta}\big(\Lambda g-\frac{\alpha}{\beta}Ric\big)(\sigma'(t),\sigma'(t))\theta(t).
\end{eqnarray*}
If $\nabla u(x)=0$, then we would have that $\theta(t)$ is the solution for the initial value problem $\theta''(t)=f(t)\theta(t)$, with $\theta'(0)=0$ and $\theta(0)=0$. Hence, we would have that $u$ vanishes along $\sigma$, which is a contradiction.

Since $\Sigma=u^{-1}(0)$ we can choose $\nu=-\frac{\nabla u}{|\nabla u|}$ to be a unit normal vector field on $\Sigma.$ Note that from \eqref{1-1}, we have $\nabla^2u=\gamma g$ on $\Sigma.$ Thus, for any vector fields $X,Y\in\mathfrak{X}(\Sigma)$, we get
\begin{equation}\label{Cal-A}
\gamma g_{\Sigma}(X,Y)=\nabla^{2}u(X,Y)=-\langle \nabla_{X}Y,\nabla u\rangle=|\nabla u|\langle \nabla_{X}Y,\nu\rangle= - |\nabla u|\mathcal{A}(X,Y),
\end{equation}
where $\mathcal{A}$ is the second fundamental form of $\Sigma.$ From identity~\eqref{Cal-A} and the fact that $\nabla u\neq 0$ on $\Sigma$, we conclude (i).

For assertion (ii), we again use that: $\nabla^2u=\gamma g$ on $\Sigma$ and $\nabla u$ is normal to $\Sigma$, hence
\begin{equation*}
X(|\nabla u|^{2})=2\nabla^{2}u(X,\nabla u)=0,
\end{equation*}
for all $X\in\mathfrak{X}(\Sigma).$ So, we conclude that the restriction of $|\nabla u|=:\xi_i$ to each   connected component $\Sigma_i$ of $\Sigma$ is a positive constant.
\end{proof}

Next, we give some examples of gradient Einstein-type metrics on geodesic balls in simply connected space forms.

\begin{example}\label{hemisferio}
Let $\mathbb{S}^{n}_{+}\subset\mathbb{R}^{n+1}$ be the upper hemisphere whose boundary is the unitary round sphere $\mathbb{S}^{n-1}.$ The Euclidean metric $g$ is an Einstein-type metric on $\mathbb{S}^{n}_{+}.$ Indeed, given a constant $\alpha$ and nonzero constants $\beta$ and $\mu$. Consider $\Lambda=\frac{\alpha}{\beta}(n-1)-\frac{\beta}{\mu}$ and the height function $h_{v}(x)=\langle \Vec{x},v\rangle$ with respect to vector $v=e_{n+1}.$ Now, it is enough to note that $h_{v}>0$ in $\mbox{int}(\mathbb{S}^{n}_{+})$, $h_{v}=0$ on $\mathbb{S}^{n-1}$, $Ric_{g}=(n-1)g$ and $\nabla^{2}h_{v}=-h_{v}g,$ moreover, $\gamma$ must be zero.
\end{example}

\begin{example}\label{geodesicballRn}
Let $B^{n}$ be the round unitary geodesic ball in $\mathbb{R}^{n}$. It is trivial see that the Euclidean metric is an Einstein-type metric on $B^{n}.$ For this, consider the function $u(x)=\frac{1}{2}-\frac{1}{2}|x|^{2}.$ Thus, given a constant $\alpha$ and nonzero constants $\beta$ and $\mu$, we can choose $\Lambda=0$ and $\gamma=-1$.
\end{example}

\begin{example}\label{geodesicballsn}
Let $\mathbb{S}^n$ be the unitary round sphere with its canonical metric $g$, and take a point $p\in\mathbb{S}^n.$ Let $M^{n}$ be the geodesic ball $B(p,r)$ of radius $r<\frac{\pi}{2}$ centered at $p$ with boundary $\Sigma=\partial B(p,r).$ Consider the radial smooth function 
\begin{equation*}
u(exp_p(sv))=\frac{\cos s}{\cos r}-1,\quad with \quad 0\leq s\leq r\quad  and \quad |v|=1.
\end{equation*}
Note that $u>0$ in $int(M)$ and $u=0$ on $\Sigma$ (where $s=r$). Thus, given a constant $\alpha$ and nonzero constants $\beta$ and $\mu$, we can choose constants $\Lambda$ and $\gamma$ such that $g$ be an Einstein-type metric on $M^n.$ In fact, since $Ric=(n-1)g$, from \eqref{1-1} we have
\begin{equation*}
-\frac{\cos s}{\cos r}=\frac{\mu}{\beta}\left(\frac{\cos s}{\cos r}-1\right)\Big(\Lambda-\frac{\alpha}{\beta}(n-1)\Big)+\gamma
\end{equation*}
and as $\gamma$ is a constant, we can apply the previous equation on $\Sigma$ to get $\gamma=-1$. We also obtain that
\begin{equation*}
\Lambda=\frac{\alpha}{\beta}(n-1)-\frac{\beta}{\mu}.    
\end{equation*}
Furthermore, consider the level sets of $u$, which we denote by $\Sigma_{s}$, $0<s\leq r$. By a straightforward computation, the mean curvature of $\Sigma_s$ is $H_s=(n-1)\cot s$, besides, the area element $A(s)$ of $\Sigma_{s}$ satisfies the initial value problem
\begin{equation*}
A'(s)=(n-1)\cot s A(s), \quad A(0)=0.
\end{equation*}
\end{example}

\begin{example}\label{geodesicballHn}
Let $\mathbb{H}^{n}$ be the hyperbolic space with its canonical metric $g$, and take a point $p\in \mathbb{H}^{n}.$ If $M^{n}$ is a geodesic ball $B(p,r)$ of radius $r$ centered at $p$, then, in an analogous way to Example~\ref{geodesicballsn}, we can prove that $g$ is an Einstein-type metric on $M^{n}$. For it, define the radial smooth function
\begin{equation*}
u(exp_p(sv))=1-\frac{\cosh s}{\cosh r}, \quad  with \quad 0\leq s\leq r\quad and \quad |v|=1,
\end{equation*}
so that $u>0$ in $int(M)$ and $u=0$ on $\Sigma$ (where $s=r$). Take a constant $\alpha$ and two nonzero constants $\beta$ and $\mu$. Next, note that $Ric=-(n-1)g$, hence 
\begin{equation*}
\gamma=-1 \quad and \quad \Lambda= \frac{\beta}{\mu}-\frac{\alpha}{\beta}(n-1).    
\end{equation*}
Furthermore, denote the level sets of $u$ by $\Sigma_{s}$, $0<s\leq r.$ By a straightforward computation, the mean curvature of $\Sigma_s$ is $H_s=(n-1)\coth s$, besides, the area element $A(s)$ of $\Sigma_{s}$ satisfies the initial value problem
\begin{equation*}
A'(s)=(n-1)\coth s A(s), \quad A(0)=0.
\end{equation*}
\end{example}
These four examples are Einstein manifolds (i.e., $Ric=\frac{S}{n}g$) with connected boundaries and are inspired by two particular types of metrics discussed in our introduction, namely, static and V-static metrics. We highlight that a vast area of research has tried to obtain the rigidity of such metrics to one of these examples, given certain geometric constraints. For instance, a long-standing and famous conjecture proposed by Boucher-Gibbons-Horowitz~\cite{BGH} asks if the only $n$-dimensional compact static metric with positive scalar curvature and connected boundary is given by a standard round hemisphere from Example~\ref{hemisferio} by taking $(\alpha,\beta,\mu)=(1,1,-1)$. See also \cite{Ambrozio} and \cite{MT2} to know some rigidity results on static and $V$-static metrics, respectively. To obtain examples having disconnected boundary, we refer the reader to~\cite[Theorem~1]{Ambrozio}. For more details on Examples~\ref{geodesicballsn} and \ref{geodesicballHn}, see proof of Proposition~\ref{sinalgamma}. 

\section{Einstein manifold case}

It is known that the existence of concircular fields in Riemannian manifolds provides an important geometric assumption. The term concircular is originated from conformal transformations that preserve the geodesic circles, called concircular transformations themselves, see Yano~\cite{yano2}. Indeed, this condition has been heavily studied, especially in the case of manifolds without boundary, see, e.g., Tashiro~\cite{Tashiro}. In the Riemannian manifolds with nonempty boundaries, some similar problems also have been considered as, e.g., in the Reilly's work~\cite{reilly1,reilly2}. A remarkable fact is that, when $\alpha=0$ and $\Lambda$ is a constant, the gradient of the function $u$ provides a concircular field, see \eqref{1-1}. In a more general way, under the Einstein assumption, we prove that this function provides a special scalar concircular field on $(M^n,g)$. After this, we work to prove the rigidity of hemispheres and geodesic balls on the class of gradient Einstein-type manifolds which are Einstein manifolds.

\begin{lemma}\label{lemma1}
Let $(M^n,g)$ be a gradient Einstein-type manifold with boundary $\Sigma$. If $(M^n,g)$ is an Einstein manifold, then the function $u$ provides a special concircular field on $(M^n,g)$. More precisely, 
\begin{equation}\label{EqConCirc}
\nabla^2 u = \Big(-\frac{S}{n(n-1)}u+\gamma\Big)g,
\end{equation}
with $u>0$ in $int(M)$ and $u=0$ on $\Sigma.$
\end{lemma}
\begin{proof}
If $(M^n,g)$ is an Einstein manifold, then from~\eqref{1-1} we have
\begin{equation}\label{1-4}
\nabla^2 u = \left[\frac{\mu}{\beta}\Big(\Lambda u-\frac{\alpha}{\beta}\frac{S}{n}u\Big)+\gamma\right] g \quad\mbox{and}\quad \Delta u = \left[\frac{\mu}{\beta}\Big(\Lambda u-\frac{\alpha}{\beta}\frac{S}{n}u\Big)+\gamma\right]n.
\end{equation}
Using the classical Bochner's formula
\begin{equation*}
Ric(\nabla \psi)+\nabla\Delta \psi =  div\nabla^2\psi
\end{equation*}
and from \eqref{1-4} we obtain
\begin{equation*}
\frac{S}{n}\nabla u+\frac{\mu}{\beta}\nabla\Big(\Lambda n u-\frac{\alpha}{\beta}S u\Big)= \frac{\mu}{\beta}\nabla\Big(\Lambda u-\frac{\alpha}{\beta}\frac{S}{n}u\Big).
\end{equation*}
Then, by connectedness of $M^n$,
\begin{equation*}
\frac{S}{n} u+\frac{\mu}{\beta}\left[\Lambda(n-1) -\frac{\alpha}{\beta}\frac{(n-1)S}{n}\right]u = C,
\end{equation*}
for some constant $C$, and as $u=0$ on $\Sigma$, then $C$ must be zero. Thus, since $u>0$ in $\mbox{int}(M)$, we get
\begin{equation}\label{eqespecial}
\frac{\mu}{\beta}\Big(\Lambda-\frac{\alpha}{\beta}\frac{S}{n} \Big)=-\frac{S}{n(n-1)}.
\end{equation}
Replacing \eqref{eqespecial} in the first equation of \eqref{1-4} we conclude the proof.
\end{proof}

Now, let $p\in M^n$ be an interior maximum point of $u$ so that $\nabla u(p)=0$, and take a unit speed geodesic $\sigma(s)$ emanating from $p$.
Moreover, without loss of generality, we assume $Ric=\kappa (n-1)g$, $\kappa\in\left\{-1,0,1\right\}$, in Lemma~\ref{lemma1} to see that
\begin{equation*}
u''(\sigma(s))=\nabla^2u(\sigma'(s),\sigma'(s))=-\kappa u(\sigma(s))+\gamma, \quad u(\sigma(0))=u(p).
\end{equation*}
Solving this initial value problem, one has
\begin{equation}\label{edo}
u(\sigma(s))=\left\{\begin{array}{lll}
\frac{\gamma}{2}s^{2}+u(p),& if & \kappa=0,\\
(u(p)-\gamma)\cos s+\gamma,& if & \kappa=1,\\
(u(p)+\gamma)\cosh s-\gamma,& if & \kappa=-1.
\end{array}\right.
\end{equation}
Take $r_{0}=d(p,\Sigma)$ and the geodesic ball $B(p,r_0)\subset M^n$ of radius $r_0$ centered at $p$, recall that $r_0$ must be less than $\pi$ for $\kappa=1$ case, see \cite[Theorem~11.16]{lee}. For each $s_{0}\in(0,r_{0}]$ consider the geodesic sphere $\Sigma_{s_{0}}$ of radius $s_{0}$. In the next result, we characterize the constant $\gamma$ by means of $u(p)$, and compute the mean curvature $H_{s_0}$ of each $\Sigma_{s_{0}}.$
\begin{proposition}\label{sinalgamma}
Let $(M^n,g)$ be a gradient Einstein-type manifold with boundary $\Sigma$. If $(M^n,g)$ is an Einstein manifold with $Ric=\kappa(n-1)g$, $\kappa\in\left\{-1,0,1\right\}$, and $p\in M^n$ is an interior maximum point of $u$, then
$$\left\{\begin{array}{lll}
\gamma< 0& if & \kappa=0,\\
\gamma<u(p)& if & \kappa=1,\\
\gamma<-u(p)& if & \kappa=-1.
\end{array}\right.$$
Furthermore, the mean curvature of each $\Sigma_{s_{0}}$ is
$$H_{s_{0}}=\left\{\begin{array}{lll}
\frac{n-1}{s_{0}}& if & \kappa=0,\\
(n-1)\cot s_{0}& if & \kappa=1,\\
(n-1)\coth s_{0}& if & \kappa=-1.
\end{array}\right.$$
\end{proposition}

\begin{proof} 
For the first part, we initially consider the $\kappa=1$ case. Since $p$ is a maximum point of $u$, $r_0<\pi$, and analysing $\nabla u(\sigma(s))$ from \eqref{edo} we get the strict inequality
\begin{equation*}
u(p)>u(\sigma(s))=(u(p)-\gamma)\cos s+\gamma,
\end{equation*}
for all $s\in(0,r_{0}].$ So,
\begin{equation*}
u(p)(1-\cos s)>\gamma(1-\cos s),    
\end{equation*}
and as $1-\cos s>0$, we conclude the result. For the case of $\kappa=-1$, again from \eqref{edo} one has
\begin{equation*}
u(p)(1-\cosh s)>-\gamma(1-\cosh s),
\end{equation*}
for all $s\in(0,r_{0}]$, and as $1-\cosh s<0$, we obtain the result. The $\kappa=0$ case is also very similar.

For the second part, we consider $\Sigma_{s_{0}}$ as above, and we prove the $\kappa=-1$ case, since the other cases are obtained by similarity. From \eqref{edo} we get
\begin{equation}\label{RC}
\langle\nabla u(\sigma(s_0)),\sigma'(s_0)\rangle=(u(p)+\gamma)\sinh(s_{0})\neq 0,    
\end{equation}
for all $s_{0}\in(0,r_{0}]$. It follows from Gauss lemma that $\Sigma_{s_{0}}$ is a hypersurface of $M^n$ and the vector fields $\nabla u(\sigma(s_0))$ and $\sigma'(s_0)$ are proportional, besides, $\nu=-\frac{\nabla u(\sigma(s_0))}{|\nabla u(\sigma(s_0))|}$ is a unit normal vector field on $\Sigma_{s_{0}}.$ Let $\mathcal{A}_{s_0}$ be the second fundamental form of $\Sigma_{s_0}.$ By Lemma~\ref{lemma1} one has $\nabla^{2}u=(u+\gamma)g$, then we conclude (analogous to~\eqref{Cal-A}) that
\begin{equation*}
\mathcal{A}_{s_0}(X,Y)=-|\nabla u(\sigma(s_0))|^{-1}(u(\sigma(s_{0}))+\gamma)g_{\Sigma_{s_0}}(X,Y),
\end{equation*}
for all $X,Y\in\mathfrak{X}(\Sigma_{s_{0}})$. Hence
\begin{equation}\label{meancurvature}
H_{s_{0}}=-(n-1)|\nabla u(\sigma(s_0))|^{-1}(u(\sigma(s_{0}))+\gamma).
\end{equation}
As $\nabla u(\sigma(s_0))$ and $\sigma'(s_0)$ are proportional, and using the characterization of $\gamma$ by means $u(p)$ obtained in the first part of this proposition, we have from \eqref{RC} that
\begin{equation}\label{mean1}
|\nabla u(\sigma(s_0))|=|u(p)+\gamma|\sinh(s_{0})=-(u(p)+\gamma)\sinh(s_{0}).
\end{equation}
On the other hand, by \eqref{edo} we have
\begin{equation}\label{mean2}
u(\sigma(s_{0}))+\gamma=(u(p)+\gamma)\cosh s_{0}.
\end{equation}
Replacing \eqref{mean1} and \eqref{mean2} into \eqref{meancurvature}, we obtain the desired result.
\end{proof}

Having completed our preliminary discussion, we are now ready to prove a complete classification of gradient Einstein-type manifolds that are Einstein manifolds. Indeed, the Einstein case is not surprising, since by dint of Lemma~\ref{lemma1} the Hessian is proportional to the metric, which is well known to be a very strong property, forcing the manifold to be a warped product. Formally, take $p$ an interior maximum point of $u$ as early and define
$$\bar{u}(x)=\frac{u(x)-u(p)}{\gamma-\kappa u(p)},$$
which is possible in view of Proposition \ref{sinalgamma}. So,
\begin{equation*}
\nabla^{2}\bar{u}=(-\kappa \bar{u}+1)g=\lambda g.
\end{equation*}
Moreover $\bar{u}(p)=0$, $d\bar{u}|_{p}=0$ and $\lambda(p)=1$. Now, we are in a position to use the Brinkmann's result~\cite{BR} (see specifically the proof of Theorem 4.3.3 in the Petersen's book \cite{Petersen}) to guarantee that, in a neighbourhood of $p$, the metric assumes the form
$$g=dr^{2}+\rho^{2}(r)ds^{2}_{n-1},$$
where 
$$\rho(r)=\bar{u}'(r)=\frac{u'(r)}{\gamma-\kappa u(p)}.$$
Thus, \eqref{edo} implies
$$g=\left\{\begin{array}{lll}
dr^{2}+r^{2} ds^{2}_{n-1}& if & \kappa=0,\\
dr^{2}+\sin^{2}r  ds^{2}_{n-1}& if & \kappa=1,\\
dr^{2}+\sinh^{2}r ds^{2}_{n-1}& if & \kappa=-1.
\end{array}\right.$$
This property extends to the entire manifold due to analyticity. In fact, similar equations to \eqref{1-1} can be recast as an elliptic system for $(u,g)$ in harmonic coordinates, implying that $u$ and $g$ are analytic. This result is well known and can be found in \cite[Theorem 5.26]{besse} (for Einstein metrics), \cite[Proposition 2.4]{HPW1} (for $m$-quasi Einstein metrics), and \cite[Proposition 2.8]{corvino} (for static metrics). In either case, this immediately proves the next two theorems that, for the sake of convenience, we write separately in the cases $\gamma\neq 0$ and $\gamma=0$ as follows.

\begin{theorem}\label{teogammaneq0}
Let $(M^{n},g)$ be a compact gradient Einstein-type manifold with connected boundary $\Sigma$. If $(M^{n},g)$ is an Einstein manifold, and $\gamma\neq 0$, then it is isometric to a geodesic ball in a simply connected space form.
\end{theorem}

Notice that Theorem~\ref{teogammaneq0} is an extension to a more general class of metrics of a known result in the context of $V$-static metrics by Miao and Tam~\cite{MT2}.

\begin{theorem}\label{gamma0}
Let $(M^{n},g)$ be a compact gradient Einstein-type manifold with connected boundary $\Sigma.$ If $(M^{n},g)$ is an Einstein manifold, and $\gamma=0$, then it is isometric to a hemisphere of a round sphere. 
\end{theorem}

It is worth mentioning that Examples~\ref{geodesicballRn}, \ref{geodesicballsn} and \ref{geodesicballHn} are explicit cases where Theorem~\ref{teogammaneq0} manifests, while for $\gamma=0$,  Example~\ref{hemisferio} is an explicit case where Theorem~\ref{gamma0} manifests.

The next step is to study conditions for a gradient Einstein-type manifold to be an Einstein manifold. This is the content of Sections~\ref{RbtEtm} and \ref{Sec-DNDC}. In this case, we immediately obtain from Theorems~\ref{teogammaneq0} and \ref{gamma0} rigidity results for geodesic balls and hemispheres by considering connectedness of the boundary.

\section{Rigidity and boundary topology of Einstein-type manifolds}\label{RbtEtm}

Our purpose here is, under some geometric assumptions, to prove rigidity results on a special class of Einstein-type manifolds. The first step is to establish an integral identity that provides a relation between the geometry of $(M^n,g)$ and its boundary $\Sigma$. We start by observing that
\begin{align*}
\stackrel\circ{\nabla^2u}&=\nabla^{2}u-\frac{\Delta u}{n}g =\frac{\mu}{\beta}u\big(\Lambda g-\frac{\alpha}{\beta}Ric\big)+\gamma g-\big[\frac{\mu}{\beta}u\big(\Lambda n-\frac{\alpha}{\beta}S\big)+\gamma n\big]\frac{g}{n}.
\end{align*}

If we define $\stackrel\circ{Ric}=Ric-\frac{S}{n}g,$ then the previous equation becomes
\begin{equation}\label{conformal}
\stackrel\circ{\nabla^2}u=-\frac{\alpha\mu}{\beta^{2}}u\stackrel\circ{Ric.}
\end{equation}

Now, we consider $\Sigma$ as the union $\cup_{i}\Sigma_{i}$ of its connected components $\Sigma_{i}.$ By Proposition~\ref{prop1}, we know that each function $\xi_{i}=|\nabla u|$ is a positive constant on $\Sigma_{i}.$ Let us denote the outward pointing unit normal vector field by $\nu,$ which satisfies $\nu=-\frac{\nabla u}{\xi_i}$ on $\Sigma_{i}.$

Notice that in the case of constant scalar curvature S, without loss of generality, we assume $S=\kappa n(n-1)$, $\kappa\in\left\{-1,0,1\right\}.$
\begin{proposition}\label{thmA}
Let $(M^n,g)$ be a compact gradient Einstein-type manifold of constant scalar curvature $S=\kappa n(n-1)$, $\kappa\in\left\{-1,0,1\right\}$, and with boundary $\Sigma = \cup_{i}\Sigma_{i}$. Then
\begin{equation*}
\sum_i\xi_i\int_{\Sigma_i}\big({\rm Ric}(\nu,\nu)-\kappa(n-1)\big)d\Sigma_i =\frac{\alpha\mu}{\beta^2}\int_{M}u\|\stackrel\circ{Ric}\|^2 dM.
\end{equation*}
\end{proposition}

\begin{proof}
First, we proceed as in Gomes~\cite{bjng} to prove the identity
\begin{equation}\label{EqMain}
div\big(\stackrel\circ{Ric}(\nabla u)\big)=\frac{n-2}{2n}\mathscr{L}_{\nabla u}S-\frac{\alpha\mu}{\beta^2}u\|\stackrel\circ{Ric}\|^2,
\end{equation}
which is true for any gradient Einstein-type metric (\ref{1-1}) with both $\beta$ and $\mu$ nonzero. Indeed, from the definition of divergence, we have
\begin{equation}\label{eqm1}
div\big(\stackrel\circ{Ric}(\nabla u)\big)=div(\stackrel\circ{Ric})(\nabla u) + \langle \nabla^2 u,\stackrel\circ{Ric}\rangle.
\end{equation}
From the second contracted Bianchi identity, we get
\begin{equation}\label{eqm2}
div(\stackrel\circ{Ric})(\nabla u)= \frac{n-2}{2n}\langle \nabla S,\nabla u\rangle.
\end{equation}
Since $\langle g,\stackrel\circ{Ric}\rangle=0$, from \eqref{conformal} we get
\begin{equation}\label{eqm3}
\langle \nabla^2 u,\stackrel\circ{Ric}\rangle = \langle\stackrel\circ{\nabla^2u},\stackrel\circ{Ric}\rangle =-\frac{\alpha\mu}{\beta^2}u\|\stackrel\circ{Ric}\|^2.
\end{equation}
Inserting \eqref{eqm2} and \eqref{eqm3} into \eqref{eqm1}, we immediately get equation~\eqref{EqMain}.

We now assume that $S=\kappa n(n-1)$ is constant and $M^n$ is compact. So, integrating equation~\eqref{EqMain}, we have
\begin{equation*}
\frac{\alpha\mu}{\beta^2}\int_{M}u\|\stackrel\circ{Ric}\|^2=-\int_{\Sigma}\stackrel\circ{Ric}(\nabla u,\nu).
\end{equation*}
Using that $\stackrel\circ{Ric}=Ric-\frac{S}{n}g$ and $\xi_i\nu=-\nabla u$ on  ${\Sigma_{i}},$ we complete our proof. 
\end{proof}

Proposition~\ref{thmA} is important in order to obtain rigidity results for $(M^n,g)$ just using suitable hypothesis on the boundary. In particular, for $\alpha\mu>0$ one has
\begin{equation*}
\sum_i\xi_i\int_{\Sigma_i}\big(Ric(\nu,\nu)-\kappa(n-1)\big)d\Sigma_i \geq 0
\end{equation*}
with equality holding if and only if $\stackrel\circ{Ric} = 0$. Note that this previous inequality is reverse for $\alpha\mu<0$ case. Thus, we immediately get the next proposition.
\begin{proposition}
Considering the same set up as in Proposition~\ref{thmA}. Then, $(M^{n},g)$ is an Einstein manifold provided that:
\begin{itemize}
\item[(i)]$\mbox{Ric}(\nu,\nu)\leq \kappa(n-1)$ along $\Sigma$ and $\alpha\mu> 0$, or
\item[(ii)]$\mbox{Ric}(\nu,\nu)\geq \kappa(n-1)$ along $\Sigma$ and $\alpha\mu< 0$.
\end{itemize}
\end{proposition}

The next two theorems can be viewed as a Chru$\acute{s}$ciel type inequality. 

\begin{theorem}\label{corovstatic}
Let $(M^n,g)$ be a compact gradient Einstein-type manifold of constant scalar curvature $S=\kappa n(n-1)$, $\kappa\in\left\{-1,0,1\right\}$, and with boundary $\Sigma=\cup_{i}\Sigma_{i}$. If $\alpha\mu<0$ ($\alpha\mu>0$), then
\begin{equation}\label{statineq}
\sum\limits_{i}\xi_i\int_{\Sigma_i}\Big(S_{\Sigma_i}-\kappa(n-2)(n-1)-\frac{n-2}{n-1}H_i^2\Big)d{\Sigma_i}\geq 0\,(\leq 0)
\end{equation}
with equality holding if and only if $(M^n,g)$ is an Einstein manifold.
\end{theorem}

\begin{proof}
By Proposition~\ref{prop1} each $\Sigma_i\subset M$ is a totally umbilical hypersurface with mean curvature $H_{i}$. So, using Gauss equation, one has
\[
2\left(Ric(\nu,\nu)-\kappa(n-1)\right)=\kappa(n-2)(n-1)+\frac{n-2}{n-1}H_i^2-S_{\Sigma_i}.
\]
The result then follows from Proposition~\ref{thmA}.
\end{proof}

In particular, since $\Sigma$ is totally geodesic when $\gamma=0$ (see Proposition~\ref{prop1}), we have the following result. 
\begin{theorem} \label{corostatic}
Let $(M^n,g)$ be a compact gradient Einstein-type manifold of constant scalar curvature $S=\kappa n(n-1)$, $\kappa\in\left\{-1,0,1\right\}$, and with boundary $\Sigma=\cup_{i}\Sigma_{i}$. If $\alpha\mu<0$ (or $\alpha\mu>0$) and $\gamma=0$, then
\begin{equation*}
\sum\limits_{i}\xi_i\int_{\Sigma_i}\Big(S_{\Sigma_i}-\kappa(n-2)(n-1)\Big)d{\Sigma_i}\geq0\,(\leq0)
\end{equation*}
with equality holding if and only if $(M^n,g)$ is an Einstein manifold.
\end{theorem}

It is worth mentioning that Theorem~\ref{corostatic} characterizes the standard hemisphere as the only gradient Einstein-type manifold $(M^n,g)$ with constant scalar curvature $n(n-1)$ whose boundary is isometric to a unit sphere. 

\begin{remark}\label{Remark2}
Analogous results to Theorems~\ref{corovstatic} and \ref{corostatic} are obtained under the weaker assumption that $\int_{M}\mathscr{L}_{\nabla u}S\,\geq 0,$ for $\alpha\mu<0,$ or $\int_{M}\mathscr{L}_{\nabla u}S\,\leq 0,$ for $\alpha\mu>0.$ Indeed, we can proceed as in the proof of Proposition~\ref{thmA} to get
\begin{equation*}
\frac{\alpha\mu}{\beta^2}\int_{M}u\|\stackrel\circ{Ric}\|^2=\frac{n-2}{n}\int_{M} \mathscr{L}_{\nabla u}S+\sum_{i}\xi_{i}\int_{\Sigma_{i}}\stackrel\circ{Ric}(\nu,\nu).   
\end{equation*}
Now, we can argument as in the proof of Theorems~\ref{corovstatic} and \ref{corostatic} to obtain the corresponding results.
We point out that this weaker assumption is closely related with the $P$ tensor, introduced in the context of $m$-quasi Einstein manifolds with nonempty boundary by He, Petersen and Wylie, see~\cite[Proposition~5.2]{HPW1}. However, even with this weaker assumption, in the case of equality we again obtain that the scalar curvature is constant.  
\end{remark}

The next result provides an upper bound for the area boundary $|\Sigma|$ of a gradient Einstein-type manifold. In this case, we use the weaker assumption given in Remark~\ref{Remark2}.

\begin{theorem}\label{areangeneral}
Let $(M^n,g)$ be a compact gradient Einstein-type manifold with boundary $\Sigma$ and $\alpha\mu<0$. If $(\Sigma,g_\Sigma)$ is an Einstein manifold of  scalar curvature $S_\Sigma$, with $\min S_\Sigma>0$, and the scalar curvature $S$ of $(M^n,g)$ satisfies $\int_{M}\mathscr{L}_{\nabla u}S\geq 0,$ then
\begin{equation}\label{area1}
|\Sigma|^{a}\left(\min_{\Sigma} S+\frac{n}{n-1}H^{2}\right)\leq n(n-1)\frac{\max S_{\Sigma}}{\min S_{\Sigma}}\omega_{n-1}^{a},
\end{equation}
where $a=2/(n-1)$ and $\omega_{n-1}$ is the volume of the unit $(n-1)$-sphere, with equality holding if and only if $S_\Sigma$ is constant and $(M^n,g)$ is an Einstein manifold. Moreover, for $n\geq 4$ estimate~\eqref{area1} reduces to
\begin{equation}\label{area1-1}
|\Sigma|^{a}\left(\min_{\Sigma} S+\frac{n}{n-1}H^{2}\right)\leq n(n-1)\omega_{n-1}^{a}.
\end{equation}
\end{theorem}
\begin{proof}
Let $Ric_{\Sigma}$ be the Ricci tensor of the metric $g_{\Sigma}$ on $\Sigma$ and $S_{\Sigma}$ its scalar curvature. Since $\min S_{\Sigma}>0$ and $\Sigma$ is compact, we can take the positive constants $\delta=\frac{\min  S_\Sigma}{(n-1)(n-2)}$ and $\varepsilon = \frac{\max  S_\Sigma}{(n-1)(n-2)}$ so that 
\begin{equation}\label{epsilon}
\delta (n-1)(n-2)\leq S_{\Sigma}\leq \varepsilon (n-1)(n-2).
\end{equation}
Then
\begin{equation}\label{delta}
Ric_{\Sigma}=\frac{S_\Sigma}{n-1}g_\Sigma\geq \delta (n-2)g_{\Sigma}.
\end{equation}
Hence, by Bishop's theorem, see, e.g., Chavel~\cite[Theorem~6 p.~74]{Chavel}, it is true that
\begin{equation}\label{area}
|\Sigma|\leq\delta^{-\frac{1}{a}}\omega_{n-1}. 
\end{equation}
where $\omega_{n-1}$ is the volume of an $(n-1)$-dimensional unit sphere. 
Since $\Sigma$ is connected, we can define $\xi=|\nabla u||_{\Sigma}$, which is a positive constant. Thus, using Remark~\ref{Remark2} and Proposition~\ref{prop1}, we obtain
\begin{equation}\label{ineq2'}
\frac{\alpha\mu}{\beta^{2}}\int_{M}u\|\stackrel\circ{Ric}\|^2\geq \frac{1}{2}\int_{\Sigma}\xi \Big(-S_{\Sigma}+\frac{n-2}{n}S+\frac{n-2}{n-1}H^{2}\Big),
\end{equation}
where we have used the Gauss equation to compute $\stackrel\circ{Ric}(\nu,\nu)$. Then, since $\alpha\mu<0$,
\begin{equation}\label{inequality2}
\int_{\Sigma}\left(\frac{n-2}{n}S+\frac{n-2}{n-1}H^{2}\right)\leq \int_{\Sigma}S_{\Sigma}.    
\end{equation}
From \eqref{inequality2} and \eqref{epsilon} we get
\begin{equation}\label{inequality3}
\min_{\Sigma} S+\frac{n}{n-1}H^{2}\leq n(n-1)\varepsilon.
\end{equation}
Thus, by using \eqref{inequality3} and \eqref{area} we obtain \eqref{area1}. Furthermore, equality holds if and only if $\Sigma$ is an Einstein manifold (see \eqref{ineq2'}) and the result follows from Theorem~\ref{gamma0}. Moreover, if $n\geq 4$, then $S_{\Sigma}$ is constant by Schur's lemma that implies~\eqref{area1-1}. 
\end{proof}

As an application of Theorem~\ref{areangeneral}, one can obtain new characterizations for hemispheres and geodesic balls in simply connected space forms. For this purpose, it is suffices to assume that the boundary is connected and to apply Theorems~\ref{teogammaneq0} and \ref{gamma0}.

\subsection{The dimension three case}\label{3-dimension}

Here, we obtain some topological constraints and classifications for the boundary. Again, we observe that in the case of constant scalar curvature $S$, without loss of generality, we can assume that $S=6\kappa$, $\kappa\in\left\{-1,0,1\right\}.$ In what follows $\chi(\Sigma)$ stands for the Euler characteristic of $\Sigma.$

We begin show how the sign of $\alpha\mu$ can lead to topological constraints for the boundary of a three-dimensional gradient Einstein-type manifold.

\begin{proposition}\label{topsphere}
Let $(M^3,g)$ be a compact gradient Einstein-type manifold with boundary $\Sigma=\cup_{i}\Sigma_{i}$ and constant scalar curvature $S=6\kappa$, $\kappa\in\left\{-1,0,1\right\}$.

\begin{itemize}
\item[(i)] If $\gamma\neq 0$ and  $\alpha\mu<0$ (for $\kappa=-1$, additionally suppose that $H_{i}>2$ in each $\Sigma_i$),  then there exists a connected component $\Sigma_{i}$ diffeomorphic to a $2$-sphere.
\item[(ii)] If $\gamma= 0$, $\alpha\mu<0$ and $\kappa=1$, then there exists a connected component $\Sigma_{i}$ diffeomorphic to a $2$-sphere.
\item[(iii)] If $\gamma\neq 0$, $\alpha\mu>0$, $\kappa=-1$ and $H_{i}\leq 2$ in each $\Sigma_i$,  then there exists a connected component $\Sigma_{i}$ diffeomorphic to a torus.
\item[(iv)] If $\gamma= 0$, $\alpha\mu>0$ and $\kappa\in\{0,-1\}$, then there exists a connected component $\Sigma_{i}$ diffeomorphic to a torus.
\end{itemize}  
\end{proposition}
\begin{proof}
The condition $\alpha\mu<0$ together with Theorem~\ref{corovstatic} and Gauss-Bonnet theorem immediately implies
\begin{equation}\label{bghet}
 4\pi\sum\limits_{i}\xi_i\chi(\Sigma_{i})\geq \sum\limits_{i}\xi_i\left(2\kappa+\frac{1}{2}H_{i}^{2}\right)|\Sigma_{i}|.   
\end{equation}
Notice that, in all cases of $\gamma$ and $\kappa$ in (i) and (ii), we get $\chi(\Sigma_{i})>0$ for some $i$ and then $\Sigma_{i}$ is homeomorphic to a $2$-sphere. The proof of (iii) and (iv) are analogous.
\end{proof}

The previous topological restrictions for the boundary become more rigid when it is connected.

\begin{corollary}\label{coro1}
Let $(M^3,g)$ be a compact gradient Einstein-type manifold of constant scalar curvature $S=6\kappa$, $\kappa\in\left\{-1,0,1\right\}$, with $\gamma\neq 0$, $\alpha\mu<0$, and connected boundary $\Sigma$. Then $\Sigma$ is diffeomorphic to the $2$-sphere (assume $H>2$ for $\kappa=-1$) and
\[
|\Sigma|\leq 4\pi\big(\kappa+\frac{1}{4}H^2\big)^{-1}
\]
with equality holding if and only if $(M^3,g)$ is isometric to a geodesic ball in a simply connected space form.
\end{corollary}

\begin{proof}
Since $\Sigma$ is connected, Proposition~\ref{topsphere} implies that $\Sigma$ is homeomorphic to a $2$-sphere, so, $\chi(\Sigma)=2$. Replacing this in \eqref{bghet}, we obtain the required area estimate. From Theorem~\ref{corovstatic}, we conclude that equality holds if and only if $(M^3,g)$ is isometric to a geodesic ball in a simply
connected space form.
\end{proof}

The next corollary gives a topological classification and an upper bound for the area boundary. The proof follows from analogous arguments as in Corollary~\ref{coro1}.

\begin{corollary}\label{coro3}
Let $(M^3,g)$ be a compact gradient Einstein-type manifold of constant scalar curvature $S=6\kappa$, $\kappa\in\left\{-1,0,1\right\},$ with $\gamma=0$, $\alpha\mu<0,$ and connected boundary $\Sigma.$ Then,
\begin{equation*}
2\pi\chi(\Sigma)\geq \kappa|\Sigma|.    
\end{equation*}
In particular, if $\kappa=1$, then $\Sigma$ is diffeomorphic to a sphere and
\begin{equation*}
|\Sigma|\leq 4\pi
\end{equation*}
with equality holding if and only if $(M^3,g)$ is isometric to a hemisphere of a round sphere.
\end{corollary}

\begin{remark}
It is worth mentioning that similar results can be obtained in the same way as in Corollaries~\ref{coro1} and \ref{coro3} by analysing $\alpha\mu>0$.
\end{remark}

\subsection{The dimension five case}
Here, we prove an upper bound for the area boundary in terms of its Euler characteristic. In fact, using the Gauss-Bonnet-Chern formula, we have the following result.
\begin{corollary}\label{teofivedimensional}
Let $(M^5,g)$ be a compact gradient Einstein-type manifold with connected boundary $\Sigma$ and $\alpha\mu<0$. If $(\Sigma,g_\Sigma)$ is an Einstein manifold, and the scalar curvature $S$ of $(M^n,g)$ satisfies $\int_{M}\mathscr{L}_{\nabla u}S\geq 0$ and $\displaystyle\min_{\Sigma}S+\frac{5}{4}H^{2}>0$, then
\[
8\pi^2\chi(\Sigma)\geq \frac{1}{24}\left(\frac{3}{5}\min_{\Sigma}S+\frac{3}{4}H^{2}\right)^2|\Sigma|\]
with equality holding if and only if $(M^5,g)$ is an Einstein manifold.
 
\end{corollary}
\begin{proof}
Notice that we can use inequality~\eqref{inequality2}, from which we deduce
\[
\left(\frac{3}{5}\min_{\Sigma}S+\frac{3}{4}H^{2}\right)|\Sigma|\leq\int_{\Sigma}S_{\Sigma}.
\]
Using H\"older inequality, we obtain
\begin{equation}\label{equacaodimensao5}
\left(\frac{3}{5}\min_{\Sigma}S+\frac{3}{4}H^{2}\right)^2|\Sigma|\leq\int_{\Sigma}S^2_{\Sigma}.
\end{equation}
Now recall the Gauss-Bonnet-Chern formula: 
\[
8\pi^2\chi(\Sigma)= \frac{1}{4}\int_{\Sigma}\|W\|^2
+\frac{1}{24}\int_{\Sigma}S^2_{\Sigma}-\frac{1}{2}
\int_{\Sigma}\|\stackrel{\circ}{\rm Ric}_{\Sigma}\|^2,
\] 
where $W$ is the Weyl tensor of $g_\Sigma$. So, under the Einstein assumption and~\eqref{equacaodimensao5}, we get  
\[
8\pi^2\chi(\Sigma)\geq \frac{1}{24}\int_{\Sigma}S^2_{\Sigma}\geq \frac{1}{24}\left(\frac{3}{5}\min_{\Sigma}S+\frac{3}{4}H^{2}\right)^2|\Sigma|.
\]
We conclude our proof from Theorem~\ref{areangeneral}, since we can use the occurrence of the equality in \eqref{inequality2}.
\end{proof}

\begin{remark}
It is worth mentioning that similar results can be obtained in the same way as in Theorem~\ref{areangeneral} and Corollary~\ref{teofivedimensional} by analysing $\alpha\mu>0$.
\end{remark}

\section{Degenerate and nondegenerate conditions}\label{Sec-DNDC}
This section consists of two parts: degenerate and nondegenerate conditions. We start by assuming the degenerate condition: $\beta^2=(n-2)\mu\alpha$ and $\beta\neq0.$ Of course, the parameters $\mu$ and $\alpha$ are not null, so that we can consider the nonconstant smooth function $f=\frac{\beta}{\mu}\ln(u)$ in $int(M).$ Note that 
\begin{equation*}
df=\frac{\beta}{\mu u}du \quad\mbox{and}\quad \nabla^2\ln u=\frac{1}{u}\nabla^2u-\frac{1}{u^2}du\otimes du, 
\end{equation*}
thus by \eqref{1-1}, one has
\begin{eqnarray*}
\nabla^2f &=& \frac{\beta}{\mu u}\nabla^2u-\frac{\beta}{\mu u^2}du\otimes du\\
&=&\Lambda g -\frac{\alpha}{\beta}Ric -\frac{\mu}{\beta}df\otimes df + \frac{\beta}{\mu u}\gamma g.
\end{eqnarray*}
Hence, PDE~\eqref{1-1} in $int(M)$ is equivalent to
\begin{equation}\label{1-1Int}
\alpha Ric+\beta \nabla^2f+\mu df\otimes df = \Big(\Lambda\beta+\frac{\beta^2}{\mu}\gamma e^{-\frac{\mu}{\beta}f}\Big)g.
\end{equation}

In the case where the boundary is empty, we consider $\gamma=0$ and $\Lambda =\rho S +\lambda$, for some constant $\rho$ and some smooth function $\lambda$ on $M^n$, so that we recover Eq.~\eqref{1-2} by Catino et al.~\cite{CMMR}. Now we follow the approach in \cite{CMMR} in order to show an expected characterization for degenerate gradient Einstein-type manifolds. Recall that a manifold $(M^n,g)$ is conformally Einstein if its metric $g$ can be pointwise conformally deformed to an Einstein metric $\tilde g.$

\begin{lemma}\label{Justify-DC}
A gradient Einstein-type manifold $(M^n,g)$ is degenerate in $int(M)$ if and only if it is conformally Einstein in $int(M).$
\end{lemma}

\begin{proof}
If $\tilde{g}=e^{2a\varphi}g$, for some real constant $a$ and some smooth function $\varphi$ on $M^n,$ then the Ricci tensor $\tilde{Ric}$ of $\tilde{g}$ is related to $g$ by the well-known formula (see Besse~\cite{besse})
\begin{equation*}
\tilde{Ric}=Ric-(n-2)a\nabla^2\varphi+(n-2)a^2d\varphi\otimes d\varphi - [(n-2)a^2|\nabla\varphi|^2+a\Delta\varphi]g.
\end{equation*}
By choosing $\varphi=f$ satisfying \eqref{1-1Int} in $int(M)$ and $a=-\frac{\beta}{(n-2)\alpha}$, one has
\begin{equation*}
\tilde{Ric}=Ric+\frac{\beta}{\alpha}\nabla^2 f+\frac{\beta^2}{(n-2)\alpha^2}df\otimes df - \Big[\frac{\beta^2}{(n-2)\alpha^2}|\nabla f|^2-\frac{\beta}{(n-2)\alpha}\Delta f\Big]g.
\end{equation*}
Under the degenerate condition $\beta^2=(n-2)\mu\alpha$, the previous equation becomes
\begin{equation*}
\tilde{Ric}=Ric+\frac{\beta}{\alpha}\nabla^2 f+\frac{\mu}{\alpha}df\otimes df - \Big[\frac{\mu}{\alpha}|\nabla f|^2-\frac{\beta}{(n-2)\alpha}\Delta f\Big]g.
\end{equation*}
From \eqref{1-1Int}, we have
\begin{equation*}
\tilde{Ric}=\frac{1}{\alpha}\Big(\Lambda\beta+\frac{\beta^2}{\mu}\gamma e^{-\frac{\mu}{\beta}f}  - \mu|\nabla f|^2 + \frac{\beta}{n-2}\Delta f\Big)g.
\end{equation*}
On the other hand, tracing \eqref{1-1Int}, we get
\begin{equation*}
\Lambda\beta+\frac{\beta^2}{\mu}\gamma e^{-\frac{\mu}{\beta}f}=\frac{1}{n}\Big(\alpha S+\beta \Delta f+\mu |\nabla f|^2\Big).
\end{equation*}
So, by a straightforward computation
\begin{equation*}
\tilde{Ric}=\frac{1}{n}\Big(S + 2\frac{\beta}{\alpha}\frac{n-1}{n-2}\Delta f-\frac{\beta^2}{\alpha^2}\frac{n-1}{n-2}|\nabla f|^2\Big) e^{\frac{2\beta}{(n-2)\alpha}f}\tilde{g}.
\end{equation*}
Thus, $\tilde{g}=e^{\frac{-2\beta}{(n-2)\alpha}f}g$ is an Einstein metric in $int(M)$, see also \cite[Section~2]{CMMR} and \cite[Theorem~1.159]{besse}.

Conversely, if $\tilde{g}=e^{\frac{-2\beta}{(n-2)\alpha}f}g$ is an Einstein metric, i.e., $\tilde{Ric}=C\tilde{g}$ for some constant $C$, with $f$ satisfying \eqref{1-1Int} in $int(M),$ then
\begin{equation*}
Ric+\frac{\beta}{\alpha}\nabla^2 f+\frac{\beta^2}{(n-2)\alpha^2}df\otimes df = \Big[\frac{\beta^2}{(n-2)\alpha^2}|\nabla f|^2-\frac{\beta}{(n-2)\alpha}\Delta f + Ce^{\frac{-2\beta}{(n-2)\alpha}f}\Big]g.
\end{equation*}
Tracing, we obtain
\begin{equation*}
S+\frac{\beta}{\alpha}\Delta f= -\frac{\beta^2}{(n-2)\alpha^2}|\nabla f|^2 +\Big[\frac{\beta^2}{(n-2)\alpha^2}|\nabla f|^2-\frac{\beta}{(n-2)\alpha}\Delta f + Ce^{\frac{-2\beta}{(n-2)\alpha}f}\Big]n.
\end{equation*}
On the other hand, again from \eqref{1-1Int}
\begin{equation*}
S+\frac{\beta}{\alpha}\Delta f=\frac{1}{\alpha}\Big[-\mu |\nabla f|^2 + \Big(\Lambda\beta+\frac{\beta^2}{\mu}\gamma e^{-\frac{\mu}{\beta}f}\Big)n\Big].
\end{equation*}
Hence
\begin{eqnarray*}
\Lambda &=& \frac{1}{n\beta}\Big[\Big(\mu -\frac{\beta^2}{(n-2)\alpha}\Big)|\nabla f|^2\Big] +\frac{\beta}{(n-2)\alpha}|\nabla f|^2-\frac{1}{n-2}\Delta f + \frac{\alpha}{\beta}Ce^{\frac{-2\beta}{(n-2)\alpha}f}\\
&&- \frac{\beta}{\mu}\gamma e^{-\frac{\mu}{\beta}f}.
\end{eqnarray*}
Now, we choose the function
\begin{equation*}
\Lambda = \frac{\beta}{(n-2)\alpha}|\nabla f|^2-\frac{1}{n-2}\Delta f + \frac{\alpha}{\beta}Ce^{\frac{-2\beta}{(n-2)\alpha}f} - \frac{\beta}{\mu}\gamma e^{-\frac{\mu}{\beta}f},
\end{equation*}
so that $\mu -\frac{\beta^2}{(n-2)\alpha}=0$ that implies $\beta^2= (n-2)\alpha\mu$.
\end{proof}

Now we assume the nondegenerate condition: $\beta^2\neq(n-2)\mu\alpha$, with all parameters not null. Here, we use the approach by the second author in~\cite{nazareno}. First of all, we observe that \eqref{1-1} in $int(M)$ is equivalent to
\begin{equation}\label{3.5-[7]}
Ric + h\nabla^2u = \ell g
\end{equation}
where $h=\frac{\beta^2}{\alpha\mu u}$ and $\ell=\frac{\Lambda\beta}{\alpha}+\frac{\gamma\beta^2}{\alpha\mu}\frac{1}{u}.$ Thus, we can use Eq.~(3.10) in~\cite{nazareno}), namely,
\begin{equation}\label{3.10-[7]}
[\beta^2-(n-2)\alpha\mu]du\wedge d\ell=0.
\end{equation}

Eq.~\eqref{3.10-[7]} immediately shows that: If $(M^n,g)$ is a gradient Einstein-type manifold and $g$ is nondegenerate in $int(M)$, then $\nabla\ell=\psi\nabla u$, for some smooth function $\psi$ in $int(M).$

We observe that for a gradient Einstein-type manifold to be an Einstein manifold it is necessary that the function $\Lambda$ be constant, see~\eqref{eqespecial}. Besides, when $\Lambda$ is constant, a necessary condition for $\ell$ be nonconstant is that $\gamma\neq 0$.

An appropriate setting to show a rigidity result for Einstein manifolds on the class of nondegenerate gradient Einstein-type manifolds $(M^n,g)$ is by means of equations~\eqref{3.5-[7]} and \eqref{3.10-[7]} both in $int(M)$ with a nonconstant function $\ell$. This is the content of the main theorem of this section. With Lemma~\ref{Justify-DC} in mind, we observe that the nondegenerate assumption in this theorem is really needed, due to the existence of homogeneous conformally Einstein metrics which are not Einstein metrics, see Besse~\cite{besse}.

\begin{theorem}\label{HCGEtM}
Let $(M^n,g)$ be a compact gradient Einstein-type manifold with constant Ricci curvatures, $\alpha$ and $\gamma$ nonzero and $\Lambda$ being a constant. If $g$ is nondegenerate in $int(M)$ and $\beta^2\neq-\mu\alpha,$ then $(M^n,g)$ is an Einstein manifold. In particular, it is isometric to a geodesic ball in a simply connected space form when its boundary is connected.
\end{theorem}

\begin{proof} 
First note that equations~$(3.11)$, $(3.12)$ and $(3.13)$ of~\cite{nazareno} still hold in $int(M)$. Since $\ell$ is nonconstant, we are in position to apply the same argument as in \cite[Theorem~3]{nazareno} to conclude that $\|\mathring{Ric}\|^2$ vanishes in $int(M)$, and the result follows by continuity. For the sake of completeness we shall present a brief sketch of the last claim. Indeed, from the second contracted Bianchi identity and \eqref{3.5-[7]}, we obtain
\begin{equation*}
\frac{\beta^2+\mu\alpha}{\mu\alpha}Ric(\nabla u) = -(n-1)u\nabla\ell -[(n-1)\ell-S]\nabla u.
\end{equation*}
Nondegenerate condition implies that $\nabla\ell=\psi\nabla u$ and then
\begin{equation*}
\frac{\beta^2+\mu\alpha}{\mu\alpha}Ric(\nabla u) = -[(n-1)u\psi+(n-1)\ell-S]\nabla u.
\end{equation*}
So, the assumption $\beta^2\neq-\mu\alpha$ implies that $\nabla u$ is an eigenvector of the Ricci tensor with constant eigenvalue, since we have assumed that the Ricci curvatures of $(M^n, g)$ are constant.
Hence,
\begin{equation*}
\mathring{Ric}(\nabla u)=C\nabla u, \quad C= -\frac{\mu\alpha}{\beta^2+\mu\alpha}[(n-1)u\psi+(n-1)\ell-S]-\frac{S}{n}.
\end{equation*}
Combining this latter equation with \eqref{EqMain} and \eqref{3.5-[7]}, we get
\begin{equation*}
C(n\ell-S)=C\frac{\beta^2}{\alpha\mu u}\Delta u=-\|\mathring{Ric}\|^2.
\end{equation*}
Since $\|\mathring{Ric}\|^2$ is constant and $\ell$ is nonconstant, we conclude, by simple analysis on this latter equation that $\|\mathring{Ric}\|
^2$ vanishes in $int(M),$ and then the result follows by continuity and Theorem~\ref{teogammaneq0}.
\end{proof}

Finally, we observe that by taking $\alpha=0$ and $\Lambda$ to be constant in \eqref{1-1}, we know from the Reilly's result~\cite[Theorem~B, Parts (I) and (II)]{reilly2} that:
\begin{itemize}
    \item[(i)] If $\Lambda=0$, then $\gamma$ must be nonzero by the maximum principle, and $g$ is flat; 
    \item[(ii)] If $\Lambda\neq0$ and $\gamma=0$, then $g$ is of constant sectional curvature $-\frac{\mu\Lambda}{\beta}.$
\end{itemize}

\section{Concluding remarks}
We conclude this paper by appending an alternative proof of Theorems~\ref{teogammaneq0} and~\ref{gamma0}. These proofs conceptually differs from the one we have already presented, as they make use of a Reilly's result and the Bishop-Gromov comparison theorem together with the explicit value of the  mean curvature of the boundary proved in Proposition~\ref{sinalgamma}, which are in agreement with our examples as well as the boundary characterization proved in Proposition~\ref{prop1}.

\subsection{An alternative proof of Theorem~\ref{teogammaneq0}}
Without loss of generality, we can suppose that $Ric=\kappa (n-1)g$, with $\kappa\in\left\{-1,0,1\right\}$. For $\kappa=0$, Lemma~\ref{lemma1} guarantees that $\nabla^{2} u=\gamma g,$ thus, from the Reilly's result~\cite{reilly2}, we know that $(M^{n},g)$ is isometric to a geodesic (metric) ball in a Euclidean space.

For $\kappa\in\{-1,1\}$, we follow the approach by Miao and Tam~\cite{MT2}. For it, let $p\in M^n$ be an interior maximum point of $u$, consider $r_{0}=d(p,\Sigma)=d(p,q_{0})$, for some $q_{0}\in\Sigma$, and a geodesic ball $B(p,r_0)$ of radius $r_0$ centered at $p$. First we claim that $\partial B(p,r_0)\subset \Sigma$. Indeed, if $q\in\partial B(p,r_0)$, then $q=\sigma(r_{0})$ for some minimizing unit speed geodesic $\sigma:[0,r_{0}]\to M^n$ emanating from $p$. Luckily, by the solution of EDO~\eqref{edo} the value $u(\sigma(s))$ depends only on the parameter $s$ (it does not depend on geodesic $\sigma$), then $u(q)=u(\sigma(r_{0}))=u(q_{0})=0$ and $q\in \Sigma=u^{-1}(0)$, this proves the claim. Since $M^n$ is compact and connected, then $M^n=B(p,r_0)$. Now, we compare the volume of $B(p,r_0)$ with the volume of a geodesic ball in a space form. Let $A(s)$ be the area element of a level surface $\Sigma_{s}\subset B(p,r_0)$, $s\in(0,r_{0}]$. We prove the $\kappa=1$ case. By Proposition~\ref{sinalgamma} and the first variation of the area formula (if necessary the reader may look the Li's book~\cite{Li}), we get the initial value problem:
\begin{equation*}
 A'(s)=(n-1)\cot s A(s) \ \ \mbox{and} \ \ A(0)=0.   
\end{equation*}
By unicity of the solution of this ODE, we have that the area of each level surface $\Sigma_{s}$ is equal to the area of a level surface in a geodesic ball in $\mathbb{S}^{n},$ see Example~\ref{geodesicballsn}. Since,
\begin{equation*}
vol(B(p,r_0))=\int_{0}^{r_{0}}A(s)ds,    
\end{equation*}
we obtain that this volume is equal to the volume of the latter mentioned geodesic ball. Using Bishop-Gromov's comparison theorem (see, e.g., Lee~\cite[p.~1171]{lee}) we conclude our proof. The $\kappa=-1$ case is very similar.

\subsection{An alternative proof of Theorem~\ref{gamma0}}
Taking $\gamma=0$ in Lemma~\ref{lemma1}, we get $\nabla^{2} u=-\frac{S}{n(n-1)} u g.$  It follows immediately from Proposition~\ref{sinalgamma} that $S>0$. Alternatively, since $u$ is nonconstant we have that $S$ is a nonzero eigenvalue of the Laplacian with Dirichlet boundary condition and, therefore, $S>0.$ As $u=0$ on $\Sigma$ (which is totally geodesic) we can use \cite[Lemma~3]{reilly1} to conclude that $(M^n,g)$ is isometric to a hemisphere of a round sphere $\mathbb{S}^n(c),$  with sectional curvature $c=\frac{S}{n(n-1)}$.

\section*{Acknowledgements}
Allan Freitas would like to thank the hospitality of the Mathematics Department of  Università degli Studi di Torino, where part of this work was carried out (there he was supported by CNPq/Brazil Grant 200261/2022-3). He is grateful to Luciano Mari for his warm hospitality and constant encouragement. Allan Freitas has also been partially supported by CNPq/Brazil Grant 316080/2021-7, by the public call n.~03 Produtividade em Pesquisa proposal code PIA13495-2020 and Programa Primeiros Projetos, Grant 2021/3175-FAPESQ/PB and MCTIC/CNPq. José Gomes has been partially supported by CNPq Grant 310458/2021-8 and by FAPESP/Brazil Grant 2023/11126-7.

\end{document}